\documentclass[11pt,twoside]{amsart}

\usepackage{amssymb}
\usepackage{latexsym}

\usepackage{color}

\theoremstyle{plain}
\newtheorem{thm}{Theorem}[section]

\newtheorem{prop}[thm]{Proposition}
\newtheorem{cor}[thm]{Corollary}
\newtheorem{lemma}[thm]{Lemma}

\theoremstyle{definition}
\newtheorem{definition}[thm]{Definition}
\newtheorem{example}{Example}

\numberwithin{equation}{section}

\begin{document}

\title{Narrow orthogonally additive operators}

\author{Marat~Pliev}

\address{South Mathematical Institute of the Russian Academy of Sciences\\
str. Markusa 22,
Vladikavkaz, 362027 Russia}

\email{maratpliev@gmail.com}

\author{Mikhail~Popov}

\address{Department of Applied Mathematics\\
Chernivtsi National University\\
str.~Kotsyubyns'koho 2, Chernivtsi, 58012 (Ukraine)}

\email{misham.popov@gmail.com}




\begin{abstract}
We extend the notion of narrow operators to nonlinear maps on vector lattices. The main objects are orthogonally additive operators and, in particular, abstract Uryson operators. Most of the results extend known theorems obtained by O.~Maslyuchenko, V.~Mykhaylyuk and the second named author published in Positivity 13 (2009), pp.~459--495, for linear operators. For instance, we prove that every orthogonally additive laterally-to-norm continuous C-compact operator from an atomless Dedekind complete vector lattice to a Banach space is narrow. Another result asserts that the set $\mathcal U_{on}^{lc}(E,F)$ of all order narrow laterally continuous abstract Uryson operators is a band in the vector lattice of all laterally continuous abstract Uryson operators from an atomless vector lattice $E$ with the principal projection property to a Dedekind complete vector lattice $F$. The band generated by the disjointness preserving laterally continuous abstract Uryson operators is the orthogonal complement to $\mathcal U_n^{lc}(E,F)$.
\end{abstract}

\maketitle


\section{Introduction}

\subsection{About the paper}

Narrow operators were introduced and studied in 1990 by Plichko and the second named author \cite{PP} as a generalization of compact operators defined on symmetric function spaces. Since that, narrow operators were defined on much more general domain spaces generalizing the previous cases, like K\"{o}the function spaces \cite{FR}, vector lattices \cite{MMP} and lattice normed spaces \cite{Pl}. Now it is a subject of an intensive study (see recent monograph \cite{PR}).

Some properties of AM-compact operators are generalized to narrow operators, but not all of them. For example, a sum of two narrow operators on $L_1$ is narrow, but on an r.i. space $E$ on $[0,1]$ with an unconditional basis every operator is a sum of two narrow operators. Similar questions are very interesting and some of them are involved, see problems \cite[Chapter~12]{PR} and recent papers \cite{MP}, \cite{PSV}. In 2009 O.~Maslyuchenko, Mykhaylyuk and the second named author discovered that the sum phenomenon for narrow operators has a pure lattice nature, extending the notion to vector lattices \cite{MMP}. Namely, a sum of two regular narrow operators is always narrow in a quite general case. Since all operators on $L_1$ are regular, the sum of any two narrow operators on $L_1$ is narrow. But if $E$ is an r.i. space on $[0,1]$ with an unconditional basis then all examples of pairs of narrow operators with nonnarrow sum involve non-regular operators.

In the present paper we generalize the main results of \cite{MMP} on narrow operators to a wide class including non-linear maps called orthogonally additive operators. Our main observation here is that, in the most contexts of narrow operators we actually use their linearity only on orthogonal elements. On the other hand, all necessary background for such operators already has been built \cite{Maz-1}, \cite{Maz-2}.

Formally, the definitions of narrow and order narrow operators on vector lattices given in \cite{MMP} can be also applied to nonlinear maps. However, some natural ``small'' nonlinear maps which have to be narrow by the idea of narrowness, are nonnarrow with respect to these definitions. In this paper, we reformulate the definitions of narrow and order narrow operators in such a way that the new definitions are equivalent to the known ones for linear maps, and are substantial for nonlinear maps.

Using these ideas we prove some new results for orthogonally additive operators and abstract Uryson operators, an important subclass of orthogonally additive operators. Theorem~\ref{thm:ourmain} asserts that every orthogonally additive laterally-to-norm continuous $C$-compact operator from an atomless Dedekind complete vector lattice to a Banach space is narrow. In Theorem~\ref{th:narmod} we solve a dominated problem for order narrow abstract Uryson operators: if $E,F$ are vector lattices with $E$ atomless and $F$ an ideal of some order continuous Banach lattice then every abstract Uryson operator $T:E\to F$ is order narrow if and only if $|T|$ is. Then we prove a number of statements which concern auxiliary notions of $\lambda$-narrow and pseudo-narrow operators, the main of which theorems~\ref{th:ddddjfhhd7} and \ref{th:mainofwhich} establish that all the notions of narrow operators coincide, under some not so restrictive assumptions. Moreover, Theorem~\ref{th:mainofwhich} asserts that the set $\mathcal U_{onlc}(E,F)$ of all order narrow laterally continuous abstract Uryson operators is a band in the Dedekind complete vector lattice of all laterally continuous abstract Uryson operator from $E$ to $F$. Moreover, the orthogonal complement to $\mathcal U_{onlc}(E,F)$ equals the band generated by all disjointness preserving laterally continuous abstract Uryson operators from $E$ to $F$.

Ideologically, all the proofs from \cite{MMP} remain almost the same for the corresponding results in the present paper. However, pretty much all of them needed for reconstruction, and several (especially Lemma~\ref{le:gemor}) needed for a new approach.

\subsection{Terminology and notation}

General information on vector lattices and Banach spaces the reader can find in the books \cite{Al,Ku, LTI, LTII}.

Let $E$ be a vector lattice. An element $u > 0$ of $E$ is an \textit{atom} provided the conditions $0 \leq x \leq u$, $0 \leq y \leq u$ and $x \wedge y = 0$ imply that either $x = 0$ or $y = 0$. A vector lattice is said to be \emph{atomless} provided it has no atom. An element $y$ of a vector lattice $E$ is called a \textit{fragment} (in another terminology, a \textit{component}) of an element $x \in E$, provided $y \bot (x-y)$. The notation $y \sqsubseteq x$ means that $y$ is a fragment of $x$. Evidently, a non-zero element $x \in E$ is an atom if and only if the only fragments of $x$ are $0$ and $x$ itself. Hence, a vector lattice $E$ is atomless if each non-zero element $x \in E$ has a proper fragment $y \sqsubseteq x$, that is, $0 \neq y \neq x$. A net $(x_\alpha)_{\alpha \in \Lambda}$ in $E$ \textit{order converges} to an element $x \in E$ (notation $x_\alpha \stackrel{\rm o}{\longrightarrow} x$) if there exists a net $(u_\alpha)_{\alpha \in \Lambda}$ in $E$ such that $u_\alpha \downarrow 0$ and $|x_\beta - x| \leq u_\beta$ for all $\beta\in \Lambda$. The equality $x=\bigsqcup_{i=1}^{n}x_{i}$ means that $x=\sum_{i=1}^{n}x_{i}$ and $x_{i}\bot x_{j}$ if $i\neq j$. Two fragments $x_{1},x_{2}$ of $x$  are called \textit{mutually complemented} or $MC$, in short, if $x = x_1 \sqcup x_2$. A sequence $(e_{n})_{n=1}^{\infty}$ is called a {\it disjoint tree} on $e \in E$ if $e_{1}=e$ and $e_{n}=e_{2n}\bigsqcup e_{2n+1}$ for each $n\in\Bbb{N}$. It is clear that all $e_{n}$ are fragments of $e$.

Fix a vector lattice $E$ and $e \in E$. The set of the all fragments of $e$ we denote by $\mathfrak{F}_e$. Denote by $\Pi_e$ the collection of all finite subsets $\pi \subset \mathfrak{F}_e$ such that $e = \bigsqcup_{x \in \pi} x$. For $\pi_1, \pi_2 \in \Pi_e$ we write $\pi_1 \leq \pi_2$ provided for every $y \in \pi_2$ there exists $x \in \pi_1$ such that $y \sqsubseteq x$. Observe that $\pi_1 \leq \pi_2$ if and only if for every $x \in \pi_1$ there exists a subset $\pi \subseteq \pi_2$ such that $x = \bigsqcup_{y \in \pi} y$. Note that $\Pi_e$ is a directed set: given any $\pi_1, \pi_2 \in \Pi_e$, one has that $\pi_1 \leq \pi_3$ and $\pi_2 \leq \pi_3$, where $\pi_3 = \{x \wedge y: \,\, x \in \pi_1 \, \& \, y \in \pi_2\}$.

\subsection{Narrow operators}

The following two definitions of a narrow operator were given in \cite{MMP}, depending on whether the range space is a Banach space or a vector lattice.

\begin{definition} \label{def:0.2}
Let $E$ be an atomless vector lattice and $X$ a Banach space. A map $f: E \to X$ is called
\begin{itemize}
  \item \textit{narrow}, if for every $x \in E^+$ and every $\varepsilon > 0$ there exists $y \in E$ such that $|y| = x$ and $\|f(y)\| < \varepsilon$;
  \item \textit{strictly narrow}, if for every $x \in E^+$ there exists  $y \in E$ such that $|y| = x$ and $f(y) = 0$.
\end{itemize}
\end{definition}

\begin{definition} \label{Defnarrorder}
Let $E, F$ be vector lattices with $E$ atomless. A linear operator $T: E \to F$ is called \textit{order narrow}\index{order narrow operator} if for every $x \in E^+$ there exists a net $(x_\alpha)$ in $E$ such that $|x_\alpha| = x$ for each $\alpha$, and $T x_\alpha \stackrel{\rm o}{\to} 0$.
\end{definition}

The atomless condition on $E$ is not essential in the above two definitions, however a narrow map must send atoms to zero, so the condition serves to avoid triviality.

\subsection{Orthogonally additive operators}

This class of operators acting between vector lattices was introduced and studied in 1990 by Maz\'{o}n and Segura de Le\'{o}n \cite{Maz-1,Maz-2}, and then considered to be defined on lattice-normed spaces by Kusraev and Pliev \cite{Ku-1,Ku-2,Pl-3}. Let $E$ be a vector lattice and $F$ a real linear space. An operator $T:E\rightarrow F$ is called \textit{orthogonally additive} if $T(x+y)=T(x)+T(y)$ whenever $x,y\in E$ are disjoint. It follows from the definition that $T(0)=0$. It is immediate that the set of all orthogonally additive operators is a real vector space with respect to the natural linear operations.

Let $E$ and $F$ be vector lattices. An orthogonally additive operator $T:E\rightarrow F$ is called:
\begin{itemize}
  \item \textit{positive} if $Tx \geq 0$ holds in $F$ for all $x \in E$;
  \item \textit{order bounded} it $T$ maps order bounded sets in $E$ to order bounded sets in $F$.

\end{itemize}

Observe that if $T: E \to F$ is a positive orthogonally additive operator and $x \in E$ is such that $T(x) \neq 0$ then $T(-x) \neq - T(x)$, because otherwise both $T(x) \geq 0$ and $T(-x) \geq 0$ imply $T(x) = 0$. So, the above notion of positivity is far from the usual positivity of a linear operator: the only linear operator which is positive in the above sense is zero.

A positive orthogonally additive operator need not be order bounded. Indeed, every function $T: \mathbb R \to \mathbb R$ with $T(0) = 0$ is an orthogonally additive operator, and obviously, not each of them is order bounded.

Another useful observation is that, if $T: E \to F$ is a positive orthogonally additive operator and $x \sqsubseteq y \in E$ then $T(x) \leq T(y)$, no matter whether $y$ is positive or not.

\subsection{Abstract Uryson operators}

The classical integral abstract Uryson operator is presented in the following example.

\begin{example}
Let $(A,\Sigma,\mu)$ and $(B,\Xi,\nu)$ be $\sigma$-finite complete measure spaces, and let $(A\times B,\mu\times\nu)$ denote the completion of their product measure space. Let $K:A\times B\times\Bbb{R}\rightarrow\Bbb{R}$ be a function satisfying the following conditions\footnote{$(C_{1})$ and $(C_{2})$ are called the Carath\'{e}odory conditions}:
\begin{enumerate}
  \item[$(C_{0})$] $K(s,t,0)=0$ for $\mu\times\nu$-almost all $(s,t)\in A\times B$;
  \item[$(C_{1})$] $K(\cdot,\cdot,r)$ is $\mu\times\nu$-measurable for all $r\in\Bbb{R}$;
  \item[$(C_{2})$] $K(s,t,\cdot)$ is continuous on $\Bbb{R}$ for $\mu\times\nu$-almost all $(s,t)\in A\times B$.
\end{enumerate}
Given $f\in L_{0}(A,\Sigma,\mu)$, the function $|K(s,\cdot,f(\cdot))|$ is $\mu$-measurable  for $\nu$-almost all $s\in B$ and $h_{f}(s):=\int_{A}|K(s,t,f(t))|\,d\mu(t)$ is a well defined and $\nu$-measurable function. Since the function $h_{f}$ can be infinite on a set of positive measure, we define
$$
\text{Dom}_{A}(K):=\{f\in L_{0}(\mu):\,h_{f}\in L_{0}(\nu)\}.
$$
Then we define an operator $T:\text{Dom}_{A}(K)\rightarrow L_{0}(\nu)$ by setting
$$
(Tf)(s):=\int_{A}K(s,t,f(t))\,d\mu(t)\,\,\,\,\nu-\text{a.e.}\,\,\,\,(\star)
$$
Let $E$ and $F$ be order ideals in $L_{0}(\mu)$ and $L_{0}(\nu)$ respectively, $K$ a function satisfying $(C_{0})$-$(C_{2})$. Then $(\star)$ defines an \textit{orthogonally additive order bounded integral operator} acting from $E$ to $F$ if $E\subseteq \text{Dom}_{A}(K)$ and $T(E)\subseteq F$.
\end{example}

In \cite{Maz-1} Maz\'{o}n and Segura de Le\'{o}n introduced and studied abstract  Uryson operators, that possess the main properties of the integral  Uryson operators. More precisely, an orthogonally additive order bounded operator $T:E\rightarrow F$ between vector lattices $E$, $F$ is called an \textit{abstract Uryson} operator. For example, any   linear operator $T\in L_{+}(E,F)$ defines a positive abstract Uryson operator by $G (f) = T |f|$ for each $f \in E$. The set of all abstract Uryson operators from $E$ to $F$ we denote by $\mathcal{U}(E,F)$. Consider some more examples of abstract Uryson operators.

\begin{example} \label{Ex1}
We consider the vector space $\mathbb R^m$, $m \in \mathbb N$ as a vector lattice with the coordinate-wise order: for any $x,y \in \mathbb R^m$ we set $x \leq y$ provided $e_i^*(x) \leq e_i^*(y)$ for all $i = 1, \ldots, m$, where $(e_i^*)_{i=1}^m$ is the coordinate functionals on $\mathbb R^m$. Let $T:\Bbb{R}^{n}\rightarrow\Bbb{R}^{m}$. Then $T\in\mathcal{U}(\Bbb{R}^{n},\Bbb{R}^{m})$ if and only if there are real functions $T_{i,j}:\Bbb{R}\rightarrow\Bbb{R}$,
$1\leq i\leq m$, $1\leq j\leq n$ satisfying $T_{i,j}(0)=0$ such that
$$
e_i^*\bigl(T(x_{1},\dots,x_{n})\bigr) = \sum_{j=1}^{n}T_{i,j}(x_{j}),
$$
In this case we write $T=(T_{i,j})$.
\end{example}

\begin{example} \label{Ex2}
Let $T:l^{2}\rightarrow\Bbb{R}$ be the operator defined by
$$
T(x_{1},\dots,x_{n},\dots) = \sum_{n\in I_{x}}n\bigl(|x_{n}|-1\bigr)
$$
where $I_{x}:=\{n\in\Bbb{N}:\,|x_{n}|\geq 1\}$. It is not difficult to check that $T$ is a positive abstract Uryson operator.
\end{example}

\begin{example} \label{Ex3}
Let $(\Omega, \Sigma, \mu)$ be a measure space, $E$ a sublattice of the vector lattice $L_0(\mu)$ of all equivalence classes of $\Sigma$-measurable functions $x: \Omega \to \mathbb R$, $F$ a vector lattice and $\nu: \Sigma \to F$ a finitely additive measure. Then the map $T: E \to F$ given by $T(x) = \nu({\rm supp} \, x)$ for any $x \in E$, is an abstract Uryson operator which is positive if and only if $\nu$ is positive.
\end{example}

The set of all abstract Uryson operators from $E$ to $F$ we denote by $\mathcal{U}(E,F)$. Consider the following order on $\mathcal{U}(E,F):S\leq T$ whenever $T-S$ is a positive operator. Then $\mathcal{U}(E,F)$ becomes an ordered vector space. Actually, one can prove more.

\begin{thm}(\cite{Maz-1},Theorem~3.2.). \label{fjjjjjg}
Let $E$ and $F$ be vector lattices, $F$ Dedekind complete. Then $\mathcal{U}(E,F)$ is a Dedekind complete vector lattice. Moreover, for each $S,T\in \mathcal{U}(E,F)$ and $x\in E$ the following conditions hold
\begin{enumerate}
  \item $(T\vee S)(x)=\sup\{T(y)+S(z):\, x = y \sqcup z\}$.
  \item $(T\wedge S)(x)=\inf\{T(y)+S(z):\, x = y \sqcup z\}$.
  \item $(T)^{+}(x)=\sup\{Ty:\, y \, \sqsubseteq x\}$.
  \item $(T)^{-}(x)=-\inf\{Ty: \, y \, \sqsubseteq x\}$.
  \item $|T(x)|\leq|T|(x)$.
  \end{enumerate}
\end{thm}

We also need the following result which represents the lattice operations of $\mathcal{U}(E,F)$ in terms of directed systems.

\begin{thm}[\cite{Maz-2}, Lemma~3.2] \label{fjjjjjgg}
Let $E$ and $F$ be vector lattices, $F$ Dedekind complete. Then for all $T,S\in\mathcal{U}(E,F)$ and $x\in E$ we have that
\begin{enumerate}
  \item $\Big\{\sum_{i=1}^{n}T(y_{i})\wedge S(y_{i}): \,\,\, x=\bigsqcup\limits_{i=1}^{n}y_{i};\,n\in\Bbb{N}\Big\}\downarrow(S\wedge T)(x)$.
  \item $\Big\{\sum_{i=1}^{n}T(y_{i})\vee S(y_{i}):\,\,\,x=\bigsqcup\limits_{i=1}^{n}y_{i};\,n\in\Bbb{N}\Big\}\uparrow(S\vee T)(x)$.
  \item $\Big\{\sum_{i=1}^{n}|T(y_{i})|:\,\,\,x=\bigsqcup\limits_{i=1}^{n}y_{i};\,n\in\Bbb{N}\Big\}\uparrow|T|(x)$.
\end{enumerate}
\end{thm}

\section{Definitions of narrow and order narrow operators for orthogonally additive maps}
\label{sec2}

Let $(\Omega, \Sigma, \mu)$ be a finite atomless measure space, $1 \leq p < \infty$ and let $T: L_p(\mu) \to \mathbb R$ be the map defined by
\begin{equation} \label{star}
Tx = \|x\|^p \,\,\, \mbox{for all} \,\,\, x \in L_p(\mu).
\end{equation}
Observe that $T$ is a nonnarrow abstract Uryson operator with respect to Definition~\ref{def:0.2}. On the other hand, by the initial idea, narrow operator must contain ``small'' operators like the above, which is a rank one operator. The same properties possesses the operator from Example~\ref{Ex3}. The first example of a nonnarrow continuous linear functional was provided in \cite{MMP}. However, this functional is not order-to-norm continuous and defined on $L_\infty$, the norm of which is not absolutely continuous (order continuous). To the contrast, the operator $T$ defined by \eqref{star} is order-to-norm continuous. The following definition, being equivalent to Definition~\ref{def:0.2} for linear maps, makes this operator to be narrow.

\begin{definition} \label{def:nar1}
Let $E$ be an atomless vector lattice and $X$ a Banach space. A map $T: E \to X$ is called:
\begin{itemize}
  \item \emph{narrow at a point} $e \in E$ if for every $\varepsilon > 0$ there exist $MC$ fragments $e_1$, $e_2$ of $e$ such that $\|T(e_1)-T(e_2)\|<\varepsilon$;
  \item \emph{narrow} if it is narrow at every $e \in E$;
  \item \emph{locally narrow at $e_0$} if $T$ is narrow at every $e \in \mathcal{F}_{e_0}$;
  \item \emph{strictly narrow at} $e \in E$ if there exist $MC$ fragments $e_1$, $e_2$ of $e$ such that $T(e_1) = T(e_2)$;
  \item \emph{strictly narrow} if it is strictly narrow at every $e \in E$.
\end{itemize}
\end{definition}

Next is the corresponding new definition of an order narrow operator.

\begin{definition} \label{def:nar2}
Let $E,F$ be vector lattices with $E$ atomless. A map $T:E\to F$ is called
\begin{itemize}
  \item \emph{order narrow at a point} $e \in E$ if there exists a net of decompositions $e = f_\alpha \sqcup g_\alpha$ such that $(T(f_\alpha) - T(g_\alpha) )\overset{\rm o}\longrightarrow 0$;
  \item \emph{order narrow} if it is order narrow at every $e \in E$;
  \item \emph{locally order narrow at $e_0$} if $T$ is order narrow at every $e \in \mathcal{F}_{e_0}$.
\end{itemize}
\end{definition}

We will use the following simple observation.

\begin{prop} \label{pr:wonder}
Let $E$ be an atomless vector lattice and $X$ a Banach space (resp., vector space or vector lattice). If an orthogonally additive operator $T: E \to X$ is narrow (resp., strictly narrow or order narrow) at every $e \in E^+ \cup E^-$ then $T$ is narrow (resp., strictly narrow or order narrow).
\end{prop}

\begin{proof}
We prove for the case of a narrow operator; proofs for the rest of cases are similar. Fix any $e \in E$ and $\varepsilon > 0$. Choose decompositions $e^+ = e_1 \sqcup e_2$ è $-e^- = e'_1 \sqcup e'_2$ so that $\|T(e_1)-T(e_2)\| < \frac{\varepsilon}{2}$ and $\|T(e'_{1})-T(e'_{2})\|<\frac{\varepsilon}{2}$. Then $e = (e_1 + e'_1) \sqcup (e_2 + e'_2)$ and
\begin{align*}
\|T(e_1 + e'_1) - T(e_2 + e'_2)\| &= \|T(e_1) + T(e'_1) - T(e_2) - T(e'_2)\| \\
&\leq \|T(e_1) - T(e_2)\| + \|T(e'_1)-T(e'_2)\| < \varepsilon.
\end{align*}
\end{proof}

Remark that narrowness of an orthogonally additive operator $T$ at elements $e \in E^+$ does not imply narrowness $T$ at elements $e \in E^-$. Indeed, given any orthogonally additive operators $T_1, T_2 : E \to X$, the operator $T: E \to X$ defined by $T (x) = T_1(x^+) + T_2(-x^-)$ for all $x \in E$ is orthogonally additive. So, for $T_1 = 0$ and $T_2 = - Id$, where $Id$ is the identity operator on $E$, the orthogonally additive operator given by $T(x) = x^-$ for all $x \in E$ is narrow at all positive elements and is not narrow at all strictly negative elements.

Proofs of the following two lemmas are simple exercises.

\begin{lemma} \label{le:2}
Let $E,F$ be vector lattices with $E$ atomless. A linear operator $T:E\to F$ is narrow (strictly narrow) in the sense of Definition~\ref{def:0.2} if and only if $T$ is narrow (strictly narrow) in the sense of Definition~\ref{def:nar1}.
\end{lemma}

\begin{lemma} \label{le:2.5}
Let $E,F$ be vector lattices with $E$ atomless. A linear operator $T:E\to F$ is order narrow in the sense of Definition~\ref{Defnarrorder} if and only if $T$ is order narrow in the sense of Definition~\ref{def:nar2}.
\end{lemma}

The proof of the following lemma is almost the same as the proof of the corresponding statement for linear maps (see \cite[Lemma~3.2]{Pl} and \cite[Proposition~10.7]{PR}).

\begin{lemma} \label{le:3}
Let $E$ be an atomless vector lattice and $F$ a Banach lattice. Then every narrow abstract Uryson operator $T:V\to W$ is order narrow.
\end{lemma}

The sets of narrow and order narrow abstract Uryson operators coincide if the range vector lattice is good enough.

\begin{lemma} \label{le:4}
Let $E$ be an atomless vector lattice and $F$ a Banach lattice with an order continuous norm. Then an abstract Uryson operator $T:E\to F$ is order narrow if and only if $T$ is narrow.
\end{lemma}

\begin{proof}
Let $T$ be order narrow. Then for every $e\in E$ there exists a net of decompositions $e = f_\alpha \sqcup g_\alpha$, $\alpha \in \Lambda$ such that $(T(f_\alpha) - T(g_\alpha) )\overset{\rm o}\longrightarrow 0$. Fix any $\varepsilon>0$. By the order continuity of the norm of $F$, we can find $\alpha_{0}\in\Lambda$ such that $\|T(f_\alpha) - T(g_\alpha)\|<\varepsilon$ for every $\alpha\geq\alpha_{0}$. So, $T$ is narrow. In view of Lemma~\ref{le:3}, the converse is true.
\end{proof}

\begin{lemma} \label{le:5}
Let $E$, $F$ and $G$ be vector lattices with $E$ atomless and $F$ an order ideal of $G$. If an abstract Uryson operator $T:E\to F$ is order narrow (resp., locally order narrow) at a point $x \in E$ then $T:E\to G$ is order narrow (resp., locally order narrow) at $x$. Conversely, if an abstract Uryson operator $T:E\to F$ is such that $T:E\to G$ is order narrow (resp., locally order narrow) at a point $x \in E$ then so is $T:E\to F$.
\end{lemma}

\begin{proof}
Observe that the statement for locally order narrow operators follows from the same statement for order narrow operators, so we prove it for order narrow operators. The first part is obvious. Let $T:E\to F$ be an abstract Uryson operator such that $T:E\to G$ is order narrow at $x_{0} \in E$ and $x\in\mathcal{F}_{x_{0}}$. Choose a net of decompositions $x = y_\alpha \sqcup z_\alpha$, $\alpha \in \Lambda$ so that $(T(y_\alpha) - T(z_\alpha) ) \overset{\rm o}\longrightarrow 0$ in $G$, that is, $|T(y_\alpha) - T(z_\alpha)| \leq u_\alpha \downarrow 0$ for some net $(u_\alpha)_{\alpha \in \Lambda}$ in $G$. By the order boundedness of $T$, there is $f \in F$ such that $|Ty| \leq f$ for every $y \sqsubseteq x$. In particular, $|T(y_\alpha) - T(z_\alpha)| \leq 2f$.
Hence, $|T(y_\alpha) - T(z_\alpha)| \leq w_\alpha \downarrow 0$ where $w_\alpha = (2f) \wedge u_\alpha \in F$. So, we have that $(w_\alpha)_{\alpha \in \Lambda} \subseteq G$. Thus, $(T(y_\alpha) - T(z_\alpha))_{\alpha \in \Lambda} \overset{\rm o}\longrightarrow 0$ in $F$.
\end{proof}

\section{Laterally-to-norm continuous C-compact abstract Uryson operators are narrow}
\label{sec4}

In this section we generalize the result of O.~Maslyuchenko, Mykhaylyuk and the second named author \cite{MMP} that every order-to-norm continuous AM-compact (linear) operator is narrow (see also \cite[Theorem~10.17]{PR}). Firstly, we consider orthogonally additive operators which is more general than linear operators. Secondly, laterally-to-norm continuity is weaker than order-to-norm continuity. Finally, C-compactness is weaker than the AM-compactness.

Recall that a net $(x_\alpha)$ in a vector lattice $E$ \textit{laterally converges} to $x \in E$ if $x_\alpha \sqsubseteq x_\beta \sqsubseteq x$ for all indices $\alpha < \beta$ and $x_\alpha \overset{\rm o}\longrightarrow x$. In this case we write $x_\alpha \overset{\rm lat}\longrightarrow x$. For positive elements $x_\alpha, x$ the condition $x_\alpha \overset{\rm lat}\longrightarrow x$ means that $x_\alpha \sqsubseteq x$ and $x_\alpha \uparrow x$.

\begin{definition}
Let $E$ be a vector lattice and $X$ a Banach space. An orthogonally additive operator $T:E\to X$ is called:
\begin{enumerate}
  \item \textit{order-to-norm} continuous if $T$ sends order convergent nets in $E$ to  norm convergent nets in $X$;
  \item \textit{laterally-to-norm} continuous provided $T$ sends laterally convergent nets in $E$ to norm convergent nets in $X$;
  \item \textit{AM-compact} if for every order bounded set $M \subset E$ its image $T(M)$ is a relatively compact set in $X$;
  \item \textit{C-compact} if the set  $\{T(y): \, y \sqsubseteq x \}$ is relatively compact in $X$ for every $x\in E$.
\end{enumerate}
\end{definition}

Since the order convergence implies the lateral convergence of a net, order-to-norm continuity of a map yields its laterally-to-norm continuity. The converse is not true: the abstract Uryson operator $T$ from Example~\ref{Ex3} is laterally-to-norm continuous if $\nu$ is a countably additive measure but not order-to-norm continuous if $\nu \neq 0$.

{\begin{example}
Let $([0,1],\Sigma_{1},\mu)$ and $([0,1],\Sigma_{2},\nu)$ be two measure spaces, $E = C[0,1]$, which is a sublattice of $L_\infty(\mu)$, and $F = L_{\infty}(\nu)$. Consider the integral Uryson operator $T:E\to F$ with the kernel $K(s,t,r) = \mathbf{1}_{[0,1]}(t) \mathbf{1}_{[0,1]}(s) |r|$. Since the interval $[0,1]$
is connected, every numerical continuous function $f: [0,1] \to \mathbb R$ is an atom, that is, $f$ has no nonzero fragment and therefore $\{T(g): \, g \sqsubseteq f\}$ is a relatively compact set in $F$ for every $f\in E$. Take $u(t) = \mathbf{1}_{[0,1]}(t)$ and consider the order bounded set $D=\{f\in E:\,|f|\leq u\}$ in $E$. Then we have
$$
T(f)(s) = \int_0^1 \mathbf{1}_{[0,1]}(t) \mathbf{1}_{[0,1]}(s)|f(t)| \, d\mu(t) = \mathbf{1}_{[0,1]}(s) \int_0^1 |f(t)|\,d\mu(t).
$$
Observe that $T(D)$ is not relatively compact in $F$. Therefore $T$ is $C$-compact, but not $AM$-compact.
\end{example}}

Remark that a C-compact abstract Uryson operator $T:E \to F$ between Banach lattices $E$, $F$ with $F$  $\sigma$-Dedekind complete is AM-compact if, in addition, $T$ is uniformly continuous on order bounded subsets of $E$ \cite[Theorem~3.4]{Maz-2}.

The main result of the section is the following theorem.

\begin{thm} \label{thm:ourmain}
Let $E$ be an atomless Dedekind complete vector lattice and $F$ a Banach space. Then every orthogonally additive laterally-to-norm continuous $C$-compact operator $T: E \to F$ is narrow.
\end{thm}

The proof of Theorem~\ref{thm:ourmain} mainly repeats the proof of \cite[Theorem~5.1]{MMP}. However, now we deal with nonlinear maps, and thus we need to follow it carefully. The following lemma is known as the lemma on rounding off coefficients \cite[p.~14]{Kad-1}.

\begin{lemma} \label{le:rounding}
Let  $(x_{i})_{i=1}^{n}$ be a finite collection of vectors in a finite dimensional normed space  $F$ and let $(\lambda_{i})_{i=1}^{n}$ be a collection of reals with $0\leq\lambda_{i}\leq 1$ for each $i$. Then there exists a collection $(\theta_{i})_{i=1}^{n}$ of numbers $\theta_{i}\in\{0,1\}$ such that
$$
\Big\|\sum_{i=1}^{n}(\lambda_{i}-\theta_{i}) \, x_{i}\Big\|\leq\frac{\text{\rm dim} \, X}{2}\max_{i}\|x_{i}\|.
$$
\end{lemma}

\begin{lemma} \label{le:fkkkiryu8}
Let $E$ be an atomless Dedekind complete vector lattice, $F$ a Banach space, $T: E \to F$ an orthogonally additive laterally-to-norm continuous operator. If $e\in E$, $|e_{n}|\leq|e|$ and $e_{n}\bot e_{m}$ for each integers $n\neq m$ then
$\lim_{n\to\infty} \|T (e_n)\| = 0$.
\end{lemma}

\begin{proof}
Since $E$ is Dedekind complete, the sequence $f_{n}=\sum_{k=1}^{n}e_{k}$ laterally
converges to $f=\sum_{k=1}^{\infty}e_{k}$
Then the laterally-to-norm continuity of $T$ implies that $Tf_{n}$
converges to $Tf$ in F. The sequence $(T(f_n))_{n=1}^\infty$ is fundamental, hence
$$
\|T(f_n) - T(f_{n-1}) \| = \Bigl\|T\Big(\sum_{k=1}^{n}e_{k}\Big)-T \Big(\sum_{k=1}^{n-1}e_{k}\Big) \Bigr\|= \|T(e_n) \|
$$
implies $\lim_{n\to\infty} \|T (e_n)\| = 0$.
\end{proof}

\begin{lemma} \label{le:4.5}
Let $E$ be an atomless Dedekind complete vector lattice, $F$ a Banach space, $T: E \to F$ an orthogonally additive laterally-to-norm continuous operator, $e\in E$. Then there exist $MC$ fragments $e_{1},e_{2}$ of $e$ such that $\|T(e_1)\| = \|T(e_2) \|$.
\end{lemma}

\begin{proof}
Fix any $MC$ fragments $e_1, e_2$ of $e$. If $\|T(e_1)\| = \|T(e_2) \|$ then there is nothing to prove. With no loss of generality we may and do assume that $\|T(e_1)\| - \|T(e_2)\| > 0$. Consider the partially ordered set
$$
D = \{g \sqsubseteq e_1: \, \|T(e_1 - g)\|-\|T(e_{2}+g)\|\geq 0\}
$$
where $g_{1} \leq g_{2}$ if and only if $g_{1}\sqsubseteq g_{2}$. If $B \subseteq D$ is a chain then $g^{\star}=\vee B\in D$ by the laterally-to-norm continuity of $T$. By the Zorn lemma, there is a maximal element $g_{0}\in D$. Now we show that $\|T(e_{1}-g_{0})\|-\|T(e_{2}+g_{0})\|=0$. Suppose on the contrary that
$$
\alpha=\|T(e_{1}-g_{0})\|-\|T(e_{2}+g_{0})\|> 0.
$$
Since $E$ is atomless, we can choose  a fragment $0\neq f\sqsubseteq(e_{1}-g_{0})$ such that $\|T(f)\|<\frac{\alpha}{4}$ and $\|T(-f)\|<\frac{\alpha}{4}$ . Since $g_{0}\bot f$, $g_{0}+f\sqsubseteq e_{1}$  we have
\begin{align*}
&\|T(e_{1}-g_{0}-f)\| - \|T(e_{2}+g_{0}+f)\| = \\
&\|T(e_{1}-g_{0})+T(-f)\| - \|T(e_{2}+g_{0})+T(f)\| \geq \\
&\|T(e_{1}-g_{0})\| - \|T(-f)\| + \|T(e_2 + g_0)\| - \|T(f)\| > \frac{\alpha}{2} > 0,
\end{align*}
that contradicts the maximality of $g_{0}$.
\end{proof}

\begin{lemma} \label{le:iiyugte678}
Let $E$ be an atomless Dedekind complete vector lattice, $F$ a Banach space, $T: E \to F$ an orthogonally additive laterally-to-norm continuous operator, $e\in E$ and $(e_{n})_{n=1}^{\infty}$ be a disjoint tree on $e$. If $\|T(e_{2n})\| = \|T(e_{2n+1})\|$ for every $n\geq 1$ then
$$
\lim_{m \to \infty} \gamma_m = 0, \,\,\,\,\, \mbox{where} \,\,\,\,\, \gamma_m = \max\limits_{2^{m}\leq i<2^{m+1}} \|T(e_i)\|.
$$
\end{lemma}

\begin{proof}
Set $\varepsilon=\limsup_{m\to\infty}\gamma_{m}$ and prove that $\varepsilon = 0$, which will be enough for the proof. Suppose on the contrary that $\varepsilon>0$. Then for each $n\in\Bbb{N}$ we set
$$
\varepsilon_{n}=\limsup\limits_{m\to\infty}\max\limits_{2^{m}\leq i<2^{m+1},\,e_{i}\sqsubseteq e_{n}} \|T(e_i)\|.
$$
Hence, for each $m\in\Bbb{N}$ one has
$$
\max\limits_{2^{m}\leq i<2^{m+1}}\varepsilon_{i}=\varepsilon.\,\,\,\,\,(\star)
$$
Now we are going to construct a sequence of mutually disjoint elements $(e_{n_{j}})_{j=1}^{\infty}$ such that $\|T(e_{n_{j}})\| \geq \frac{\varepsilon}{2}$, that is impossible by Lemma~\ref{le:fkkkiryu8}. At the first step we choose $m_{1}$ so that $\max\limits_{2^{m_{1}}\leq i<2^{m_{1}+1}}\|T(e_i)\|\geq\frac{\varepsilon}{2}$. By $(\star)$, we choose $i_1$, $2^{m_{1}}\leq i_{1}<2^{m_{1}+1}$ so that $\varepsilon_{i_1}=\varepsilon$. Using $\|T(e_{2n})\| = \|T(e_{2n+1})\|$, we choose $n_1\neq i_1$, $2^{m_{1}}\leq n_1 <2^{m_{1}+1}$ so that $\|T(e_{n_{1}})\|\geq\frac{\varepsilon}{2}$. At the second step we choose $m_{2}>m_{1}$ so that
$$
\max\limits_{2^{m}\leq i<2^{m+1},\,e_i \sqsubseteq e_{i_1}}\|T(e_i)\|\geq\frac{\varepsilon}{2}.
$$
By $(\star)$, we choose $i_{2}$, $2^{m_{2}}\leq i_{2}<2^{m_{2}+1}$ so that $\varepsilon_{i_{2}}=\varepsilon$. Then we choose $m_{2}\neq i_{2}$, $2^{m_{2}}\leq i_{2}<2^{m_{2}+1}$ so that $\|T(e_{m_2})\|\geq\frac{\varepsilon}{2}$. Proceeding further, we construct the desired sequence. Indeed, $\|T(e_{m_i})\|\geq\frac{\varepsilon}{2}$ by the construction and the mutual disjointness for $e_{m_{l}},e_{m_{j}}$, $j\neq l$ is guaranteed by the condition $m_{j}\neq i_{j}$, because the elements $e_{m_{j+l}}$ are fragments of $e_{i_{j}}$ which are disjoint to $e_{m_{j}}$.
\end{proof}

\begin{lemma} \label{le:w88495737}
Let $E$ be an atomless Dedekind complete vector lattice, $F$ a finite dimensional Banach space. Then every orthogonally additive laterally-to-norm continuous operator $T: E \to F$ is narrow.
\end{lemma}

\begin{proof}
Fix any $e\in E$ and $\varepsilon>0$. Using Lemma~\ref{le:4.5}, we construct a disjoint tree $(e_{n})$ on $e$ with $\|T(e_{2n})\| = \|T(e_{2n+1})\|$ for all $n\in\Bbb{N}$. By lemma~\ref{le:iiyugte678} we choose $m$ so that $\gamma_m \, {\rm dim} \, F<\varepsilon$. Then using Lemma~\ref{le:rounding}, we choose numbers $\lambda_{i}\in\{0,1\}$ for $i=2^{m},\dots, 2^{m+1}-1$ so that
\begin{equation} \label{eq:855jfg}
\Bigl\|2\sum_{i=2^{m}}^{2^{m+1}-1} \Bigl(\frac{1}{2}-\lambda_{i} \Bigr) \, T(e_i)\Bigr\|
\leq {\rm dim} \, F \max\limits_{2^{m}\leq i<2^{m+1}} \|T(e_i)\|=
\gamma_m \, {\rm dim} \, F<\varepsilon.
\end{equation}

Observe that for $I_1 = \{i = 2^m, \ldots, 2^{m+1}-1: \,\, \lambda_i = 0\}$ and $I_2 = \{i = 2^m, \ldots, 2^{m+1}-1: \,\, \lambda_i = 1\}$ the vectors $f_j = \sum_{i \in I_j} e_i$, $j = 1,2$ are MC fragments of $e$ and by \eqref{eq:855jfg},
$$
\|T(f_1) - T(f_2)\| = \Bigl\|\sum_{i=2^{m}}^{2^{m+1}-1} (1-2\lambda_{i}) \, T(e_i)\Bigr\| < \varepsilon.
$$
\end{proof}

Now we are ready to prove the main result of the section.

\begin{proof}[Proof of Theorem~\ref{thm:ourmain}]
We may consider $F$ as a subspace of some $l_{\infty}(D)$ space
$$
F \hookrightarrow F^{\star\star} \hookrightarrow l_{\infty}(B_{F^{\star}})=l_{\infty}(D)=W.
$$
By the notation $\hookrightarrow$ we mean an isometric embedding. It is well known that if $H$ is a relatively compact subset of $l_{\infty}(D)$ for some infinite set $D$ and $\varepsilon>0$ then there exists a finite rank operator $S\in l_{\infty}(D)$ such that $\|x-Sx\|\leq\varepsilon$ for every $x\in H$ \cite[Lemma~10.25]{PR}. Fix any $e\in E$ and $\varepsilon>0$. Since $T$ is a $C$-compact operator, $K = \{T(g):\,g \,\,\text{is a fragment of}\,\, f \}$ is relatively compact in $X$ and hence, in $W$. By the above, there exist a finite rank linear operator $S\in\mathcal{L}(W)$ such that $\|f- Sf\|\leq\frac{\varepsilon}{4}$ for every $f\in K$. Then $R=S\circ T$ is an orthogonally additive laterally-norm continuous finite rank operator. By Lemma~\ref{le:w88495737}, there exist $MC$ fragments $e_{1},e_{2}$ of $e$ such that $\|R(e_1) - R(e_2)\| < \frac{\varepsilon}{2}$. Thus,
\begin{align*}
&\|T(e_1) - T(e_2) \| \\
&= \|T(e_1) - T(e_2) + S(T(e_1)) - S(T(e_2)) - S(T(e_1)) + S(T(e_2)) \| \\
&=\|T(e_{1}) - T(e_{2}) + R(e_{1}) - R(e_{2}) - S(T(e_{1})) + S(T(e_{2}))\| \\
&\leq\|R(e_{1}) - R(e_{2})\| + \|T(e_{1}) - S(T(e_{1})) - (T(e_{2}) - S(T(e_{2})))\| \\
&<\frac{\varepsilon}{2}+\frac{\varepsilon}{2}=\varepsilon.
\end{align*}
\end{proof}

\section{Domination problem for abstract Uryson narrow operators}
\label{sec5}

In this section we consider a domination problem for the modulus of abstract Uryson operators. In the classical sense, the domination problem can be stated as follows. Let $E$, $F$ be vector lattices, $S,T: E \to F$ linear operators with $0 \leq S \leq T$. Let $\mathcal P$ be some property of linear operators $R: E \to F$, so that $\mathcal P(R)$ means that $R$ possesses $\mathcal P$. Does $\mathcal P(T)$ imply $\mathcal P(S)$? Another version: if $|S| \leq T$, then whether $\mathcal P(T)$ implies $\mathcal P(S)$? See \cite{FHT} for a survey on the domination problem for ``small'' operators.

Let $E,F$ be vector lattices with $E$ atomless, $T: E \to F$ a linear operator, $E_1$ a vector sublattice of $E$, $x_0 \in E_1$ and $\mathfrak{F}_{x_0} \subseteq E_1$. Suppose $S:E_{1}\to F$ is an abstract Uryson operator with $S|_{\mathfrak{F}_{x_{0}}}=T|_{\mathfrak{F}_{x_{0}}}$. Obviously,  $S$ is locally narrow at $x_{0}$ if and only if $T$ is.

\begin{thm} \label{th:narmod}
Let $E,F$ be vector lattices, with $E$ atomless and $F$ an ideal of some order continuous Banach lattices. Then every abstract Uryson operator $T:E\to F$ is order narrow if and only if is $|T|$ is.
\end{thm}

First we need the following lemma.

\begin{lemma} \label{le:lsd7}
Let $E$ be a vector lattice, $T: E \to L_1(\mu)$ an abstract Uryson operator, $\varepsilon > 0$, $n \in \mathbb N$, $x, y_k, u_k, v_k \in E$ for $k = 1, \ldots, n$ satisfy $y_k = u_k \sqcup v_k$,
\begin{equation} \label{eq:fhhfg9d}
x=\bigsqcup_{i=1}^{n}y_{i} \,\,\, \mbox{and} \,\,\, \Big\||T|(x)-\sum_{i=1}^{n}|T(y_{i})|\,\Big\|<\varepsilon.
\end{equation}
Then
\begin{equation} \label{eq:fhhfg9dprime}
\sum_{i=1}^n \Bigl(\| |T|(u_i)-|T(u_{i})|\, \| + \bigl\| |T|(v_i)-|T(v_i)|\, \| \Bigr) < \varepsilon.
\end{equation}
\end{lemma}

\begin{proof}[Proof of Lemma~\ref{eq:fhhfg9dprime}]
Since $x = \bigsqcup_{k=1}^n (u_k \sqcup v_k)$, by (3) of Theorem~\ref{fjjjjjgg} and the orthogonal additivity of $T$ and $|T|$, one has
\begin{equation} \label{eq:fhhfg9d2}
0\leq|T|(x)-\sum_{k=1}^n \Bigl( |T(u_k)|+|T(v_k)| \Bigr) \leq |T|(x)-\sum_{k=1}^n |T(y_k) |
\end{equation}
and
\begin{equation} \label{eq:fhhfg9d3}
|T|(x) = \sum_{k=1}^n (|T|(u_k)+|T|(v_k)).
\end{equation}
Since $|T|(u_k) - |T(u_k)|$ and $|T|(v_k) - |T (v_k)|$ are positive elements of $L_1(\mu)$ for all $k\in\{1,\dots,n\}$, the sum of their norms equals the norm of their sum. Thus, taking into account \eqref{eq:fhhfg9d2} and \eqref{eq:fhhfg9d3}, we obtain
\begin{align*}
&\sum_{k=1}^n \Bigl(\| |T|(u_k)-|T(u_k)|\, \| + \| |T|(v_k)-|T(v_k)|\, \| \Bigr)\\
&= \Bigl\| |T|(x) - \sum_{k=1}^n \Bigl( |T(u_k)|+|T(v_k)| \Bigr) \Bigr\| \leq \Bigl\| |T|(x) - \sum_{k=1}^n | T(y_k) |\, \Bigr\| < \varepsilon.
\end{align*}
\end{proof}

Before the proof of Theorem~\ref{th:narmod}, we make a comment on it. The general idea is the same as in the linear case: first we consider the case of $F=L_{1}(\mu)$, and then prove in the general case using this partial case. The proof of the first part has some minor adjustments. However, the proof of the second part is quite different. We cannot use the same idea, because in the nonlinear case the image $T(I_x)$ of an order ideal $I_x$ in $E$ generated by $x \in E$ under an abstract Uryson operator $T: E \to F$ need not be contained in the order ideal $I_{T(x)}$ in $F$ generated by $T(x)$.

The next lemma presents the main tool for the proof of Theorem~\ref{th:narmod}.

\begin{lemma} \label{le:lsd8}
Let $E,F$ be vector lattices, with $E$ atomless and $F$ an ideal of some order continuous Banach lattices and $x_0 \in E$. Then every abstract Uryson operator $T: E \to F$ is locally order narrow at $x_0$ if and only if is $|T|$ is.
\end{lemma}

\begin{proof}[Proof of Lemma~\ref{le:lsd8}]
First we prove the lemma for the case of $F=L_{1}(\mu)$. By Lemma~\ref{le:5}, we consider local narrowness instead of the local order narrowness. Fix any $x_0 \in E$ and assume first that $T$ is locally narrow at $x_{0}$.  Fix any $x\in\mathcal{F}_{x_{0}}$ and $\varepsilon>0$.
By (3) of Theorem~\ref{fjjjjjgg},
$$
\Bigl\{|T|(x) - \sum_{k=1}^n | T(y_k) |: \,\,\, x=\bigsqcup\limits_{k=1}^n y_k; \,\,\, n\in\Bbb{N}\Bigr\} \downarrow 0.
$$
Hence, by the order continuity of $L_{1}(\mu)$, we can choose $y_1, \dots, y_n \in\mathcal{F}_{x}$ such that \eqref{eq:fhhfg9d} holds.

For each $k=1,\dots,n$ we represent $y_k=u_k\sqcup v_k$ so that $u_k,v_k\in E$ and $\|T(u_k)-T(v_k)\|<\varepsilon/n$. Then putting $u=\bigsqcup_{k=1}^k u_k$, $v=\bigsqcup_{k=1}^n v_k$, we obtain $x = u \sqcup v$ and by Lemma~\ref{eq:fhhfg9dprime}
\begin{align*}
&\||T|(u)- |T|(w)\|\leq\sum_{k=1}^n \||T|(u_k)-|T|(v_k)\| \leq \\
&\sum_{k=1}^n \||T(u_k)| - |T(v_k)| \| + \sum_{k=1}^n \Bigl( \||T|(u_k)-|T(u_k)| \| + \| |T| (v_k) - |T(v_k) | \| \Bigr) \\
&< \sum_{k=1}^n \| T(u_k) - T(v_k) \| + \varepsilon < 2 \varepsilon.
\end{align*}

By arbitrariness of $x\in\mathcal{F}_{x_{0}}$ and $\varepsilon>0$, $|T|$ is locally narrow at $x_{0}$.

Now let $|T|$ be locally narrow at $x_{0}$, $x \in\mathcal{F}_{x_{0}}$, $\varepsilon > 0$. Like in the first part of the proof, we choose choose $y_1, \dots, y_n \in\mathcal{F}_{x}$ so that \eqref{eq:fhhfg9d} is satisfied. For each $k=1,\dots,n$ we decompose $y_k = u_k \sqcup v_k$ so that $\| |T|(u_k) - |T| (v_k) \| < \varepsilon/n$ and set $u=\bigsqcup_{k=1}^k u_k$ and $v=\bigsqcup_{k=1}^n v_k$.

Since $|a-b| + |a+b| = |a| + |b| + | |a| - |b| |$ for all $a,b \in \mathbb R$, we have that for each $f,g \in L_1(\mu)$
\begin{equation} \label{visimshist}
\|f - g\| = \| |f| - |g| \| + \|f\| + \|g\| - \|f + g\|.
\end{equation}

Observe that since $x = \bigsqcup_{k=1}^n (u_k \sqcup v_k)$, by (3) of Theorem~\ref{fjjjjjgg},
\begin{equation} \label{visimshist001}
\Bigl\| \sum_{k=1}^n ( | T (u_k) | + | T (v_k) | ) \Bigr\| \leq \| | T | \, (x) \|,
\end{equation}
and by the triangle inequality,
\begin{align} \label{visimshist002}
& \,\,\,\,\,\,\,\, \| | T (u_k) | - | T (v_k) | \| \\
& \leq \| | T | \, (u_k) - | T | \, (v_k) \| + \| | T | \, (u_k) - | T (u_k) | \, \| + \| | T | \, (v_k) - | T (v_k) | \, \| \notag
\end{align}
for all $k = 1, \ldots, n$. Since the $L_1$-norm of a sum of positive elements equals the sum of their norms, using \eqref{visimshist} for $f = T(u_k)$, $g = T(v_k)$, \eqref{visimshist001} and \eqref{visimshist002}, we obtain
\begin{align*}
&\| T (u) - T (v) \| \leq \sum_{k=1}^n \| T (u_k) - T (v_k) \|\\
&= \sum_{k=1}^n \| | T (u_k) | - | T (v_k) | \| + \Bigl\| \sum_{k=1}^n ( | T (u_k) | + | T (v_k) | ) \Bigr\| - \Bigl\| \sum_{k=1}^n | T (y_k) | \Bigr\|\\
&\leq \sum_{k=1}^n \| | T | \, (u_k) - | T | \, (v_k) \| + \sum_{k=1}^n \Bigl( \| | T | \, (u_k) - | T (u_k) | \, \| \\
&+ \| | T | \, (v_k) - | T (v_k) | \, \|\Bigr) + \| | T | \, (x) \| - \Bigl\| \sum_{k=1}^n | T (y_k) | \Bigr\|\\
&\stackrel{\rm by \ Lemma \ \ref{eq:fhhfg9dprime}}{<} n \, \frac{\varepsilon}{n} \, + \varepsilon + \Bigl\| |T| (x) - \sum_{k=1}^n |T (y_k) | \Bigr\| < 3 \varepsilon
\end{align*}
by the choose of $y_1, \dots, y_n \in\mathcal{F}_{x}$. Thus, $T$ is locally narrow at $x_{0}$.

Now we consider the general case. Fix  $x_{0} \in E$ and any $x\in\mathcal{F}_{x_{0}}$. For each $\pi \in \Pi_x$  we set $L_\pi = {\rm linear \ span} \, \pi$. Note that $\pi_1 \leq \pi_2$ in $\Pi_x$ implies $L_{\pi_1} \subseteq L_{\pi_2}$.  In particular, $(L_\pi)_{\pi \in \Pi_x}$ is a net with respect to the inclusion. The following linear subspace is a sublattice of $E$
$$
E_1 = \bigcup_{\pi \in \Pi_x} L_\pi = {\rm linear \ span} \, (\mathfrak{F}_x).
$$
For each $\pi=(x_{i})_{i=1}^{n}\in\Pi_x$ we define the linear operator $S_\pi: L_\pi \to F$ which extends the equality $S_\pi x_i = T(x_i)$ for $i=1,\dots,n$ to $L_\pi$ by linearity. By orthogonal additivity of $T$, if $\pi_1 \leq \pi_2$ in $\mathfrak{F}_x$ then $S_{\pi_2}|_{L_{\pi_1}} = S_{\pi_1}$. Thus, we can define the following orthogonally additive operator $S: E_1 \to F$ by
$$
S(x) = \lim_{\pi \in \Pi_x} S_\pi (x), \,\, x \in E_1.
$$
By the construction of $S$,
\begin{equation} \label{verynew}
S \big|_{\mathfrak{F}_{x_0}} = T \big|_{\mathfrak{F}_{x_0}} \,\,\text{and}\,\, |S| \big|_{\mathfrak{F}_{x_0}}=|T| \big|_{\mathfrak{F}_{x_0}}
\end{equation}
(the second equality follows from (3) of Theorem~\ref{fjjjjjgg}).

By \eqref{verynew}, $T$ is locally order narrow at $x_0$ if and only if $S$ is order narrow. By Theorem~\ref{th:narmod} for linear operators (see \cite[Theorem~10.26]{PR}), $S$ is order narrow if and only if $|S|$ is order narrow. Finally, by \eqref{verynew}, $|S|$ is order narrow if and only if $|T|$ is locally order narrow at $x_0$.
\end{proof}

\begin{proof}[Proof of Theorem~\ref{th:narmod}]
Let $T\in\mathcal{U}(E,F)$ be an order narrow operator. Then $T$ is locally order narrow at every $x\in E$. By Lemma~\ref{le:lsd8}, $|T|$ is locally order narrow at every $x_0 \in E$ and therefore $|T|$ is order narrow. The converse assertion can be proved by the same arguments.
\end{proof}

Remark that, to prove Theorem~\ref{th:narmod} for the general case, we have used the case of $F=L_{1}(\mu)$ for linear operators only. However, we could not use the corresponding theorem from \cite{MMP} for linear operators, because it was proved under superfluous assumptions of Dedekind completeness of $E$ and order continuity of $T$.

\section{$\lambda$-narrow abstract Uryson operators}
\label{sec8}

In this section, by analogy with the linear case, we consider $\lambda$-narrow abstract Uryson operators acting between vector lattices.

Let $E,F$ be vector lattices with $F$ Dedekind complete. For an abstract Uryson operator $T:E \to F$ we define a function $\lambda_T: E^+ \to F^+$ by setting
\begin{equation} \label{eq:amsldfkjh7}
\lambda_{T}(x)=\bigwedge_{\pi\in\Pi_{x}}\bigvee_{y\in\pi}|T(y)|
\end{equation}
for every $x \in E^+$. Since $T$ is order bounded and $F$ is Dedekind complete, the function $\lambda_{T}$ is well defined. We say that $\lambda_T$ is the \emph{Enflo-Starbird function} of $T$.

\begin{definition}
Let $E,F$ be vector lattices with $F$ Dedekind complete and $E$ atomless. An abstract Uryson operator $T:E\to F$ is called $\lambda$\emph{-narrow} if $\lambda_{T}(x)=0$ for every $x\in E$.
\end{definition}

It is clear that $\lambda_{T}(x)\leq\lambda_{|T|}(x)$ for every abstract Uryson operator $T$ and every $x\in E^+$.
The following lemma is an analogue of Lemma~\ref{le:5} for $\lambda$-narrow operators.

\begin{lemma} \label{le:1}
Let $E,F,G$ be vector lattices such that, $E$ is atomless and $F$ is an ideal of $G$. An  abstract Uryson operator $T:E\to F$ is $\lambda$-narrow if and only if $T:E\to G$ is $\lambda$-narrow.
\end{lemma}

\begin{proof}
Fix any $x\in E^+$ and set $\lambda_{\pi}=\bigvee_{u\in\pi}|Tu|$
for each $\pi\in\Pi_{x}$ of $x$. By \eqref{eq:amsldfkjh7}, $\lambda_{T}(x)=\inf_{\pi\in\Pi_{x}}\lambda_{\pi}$.
Since $F$ is an ideal of $G$, we have that $\inf_{\pi\in\Pi_{x}}\lambda_{\pi}=0$
in $F$  if and only if $\inf_{\pi\in\Pi_{x}}\lambda_{\pi}=0$ in $G$.
\end{proof}

We also need the following known lemma.

\begin{lemma}[\cite{MMP}, \cite{PR}, Lemma~7.50] \label{le:1kkkk1}
Let $z_{1},\dots,z_{2n}\in L_{1}^{+}(\mu)$, $K=\sum_{i=1}^{2n} \|z_{i}\|$,
$z_{0}=\bigvee_{i=1}^{2n}z_{i}$ and $\alpha=\|z_{0}\|$. Then there exists a permutation
$\tau:\{1,\dots,2n\}\to\{1,\dots,2n\}$ such that
$$
\Big\|\sum_{i=1}^{2n}(-1)^{i}z_{\tau(i)}\Big\|\leq\sqrt{2\alpha K}.
$$
\end{lemma}

\begin{thm} \label{th:lamordern}
Let $E, F$ be  vector lattices with $F$ Dedekind complete, such that $E$ is atomless and $F$ is an ideal of an order continuous Banach lattice $G$. Then a positive abstract Uryson operator $T: E \to F$ is $\lambda$-narrow
if and only if $T$ is order narrow.
\end{thm}

\begin{proof}
By Lemma~\ref{le:4}, instead of the order narrowness we consider narrowness. Assume $T$ is narrow. We show that $T$ is $\lambda$-narrow. Fix any $x \in E$ and $\varepsilon>0$. It is enough to prove that there is $\pi = \{x_i: \, i=1, \ldots, 2^m\} \in \Pi_x$ such that $\|\bigvee_{i=1}^{2^m} T(x_i)\|<\varepsilon$. Choose $m$ so that $2^{-m}\|T(x)\|<\varepsilon/2$. Using that $T$ is narrow and orthogonally additive, we construct a disjoint tree $(y_k)_{k=1}^\infty$ on $x$ so that $\|T(y_{2k}) - T(y_{2k+1})\| < 2^{-m-1} \varepsilon \stackrel{\rm def}{=} \varepsilon_1$. Then we claim that
\begin{equation} \label{ec:claim}
\|T(y_{2^n+j}) - 2^{-n} T(x)\| < \varepsilon_1 \,\,\, \mbox{for all} \,\,\, n = 1,2, \ldots \, j = 0,1, \ldots, 2^n-1.
\end{equation}
Indeed, for $n=1$ and $j=0,1$ this follows from the next observation
$$
\|T(y_{2+j}) - 2^{-1} T(x) \| = \frac12 \|T(y_2) - T(y_3)\| < \varepsilon_1.
$$
Supposing the claim is true for a given $n$ and all $j = 0,1, \ldots, 2^n-1$, we obtain for $n+1$ and all $j = 0,1, \ldots, 2^n-1$
\begin{align*}
&\bigl\|T(y_{2^{n+1}+2j}) - 2^{-n-1}T(x) \bigr\| = \frac12 \bigl\|2 T(y_{2^{n+1}+2j}) - 2^{-n} T(x) \bigr\|\\
&\leq \frac12 \bigl\|2 T(y_{2^{n+1}+2j}) - T(y_{2^n+i}) \bigr\| + \frac12 \bigl\|T(y_{2^n+i}) - 2^{-n} T(x) \bigr\|\\
&< \frac12 \bigl\|T(y_{2^{n+1}+2j}) - T(y_{2^{n+1}+2j+1}) \bigr\| + \frac{\varepsilon_1}{2} < \frac{\varepsilon_1}{2} + \frac{\varepsilon_1}{2} = \varepsilon_1.
\end{align*}
Analogously, $\bigl\|T(y_{2^{n+1}+2j+1}) - 2^{-n-1}T(x) \bigr\| < \varepsilon_1$, and the claim is proved by induction.

Now set $x_i = y_{2^m+i-1}$ for $i = 1, \ldots, 2^m$. Then $x = \bigsqcup_{i=1}^{2^m} x_i$ and, by \eqref{ec:claim}, $\|T(x_i) - 2^{-m} T(x)\| < \varepsilon_1$ for $i = 1, \ldots, 2^m$. Hence,
\begin{equation} \label{eq:cnxb7}
\frac{T(x)}{2^m} - \bigvee_{i=1}^{2^m} T(x_i) \leq \frac{T(x)}{2^m} - T(x_1) \leq \sum_{i=1}^{2^m} \Bigl|\frac{T(x)}{2^m} - T(x_i)\Bigr|
\end{equation}
(in the middle part of the previous inequality one can put any of $x_k$ instead of $x_1$). On the other hand,
$$
T(x_k) = \frac{T(x)}{2^m} + \Bigl( T(x_k) - \frac{T(x)}{2^m} \Bigr) \leq \frac{T(x)}{2^m} + \sum_{i=1}^{2^m} \Bigl|\frac{T(x)}{2^m} - T(x_i)\Bigr|
$$
for every $1 \leq k \leq 2^m$. Therefore
$$
\bigvee_{k=1}^{2^m}T(x_k) \leq \frac{T(x)}{2^m} + \sum_{i=1}^{2^m} \Bigl|\frac{T(x)}{2^m} - T(x_i)\Bigr|
$$
and hence
\begin{equation} \label{eq:cnxb8}
\bigvee_{k=1}^{2^m}T(x_k) - \frac{T(x)}{2^m} \leq \sum_{i=1}^{2^m} \Bigl|\frac{T(x)}{2^m} - T(x_i)\Bigr|
\end{equation}
Then \eqref{eq:cnxb7} and \eqref{eq:cnxb8} imply
\begin{align*}
\Big\|\bigvee_{i=1}^{2^m} T(x_i) - \frac{T(x)}{2^m} \Big\| &\leq \Big\| \sum_{i=1}^{2^m} \Bigl|\frac{T(x)}{2^m} - T(x_i)\Bigr| \Big\| \\
&\leq \sum_{i=1}^{2^m} \Bigl\|\frac{T(x)}{2^m} - T(x_i)\Bigr\| < 2^m \varepsilon_1 = \frac{\varepsilon}{2}.
\end{align*}
Finally we obtain
$$
\Big\|\bigvee_{i=1}^{2^m} T(x_i) \Big\| \leq \Big\|\bigvee_{i=1}^{2^m}T(x_i) - \frac{T(x)}{2^m} \Big\| + \Bigl\| \frac{T(x)}{2^m} \Bigr\| < \frac{\varepsilon}{2}+\frac{\varepsilon}{2} = \varepsilon.
$$
Thus, $T$ is $\lambda$-narrow.

Now let $T$ be a $\lambda$-narrow operator. We prove that $T$ is narrow. Like in the
proof of Theorem~\ref{th:narmod}, we consider the case of $F=L_{1}(\mu)$, and the general case can
be reduced to this one exactly like in Theorem~\ref{th:narmod}). Fix any $x\in E$ and $\varepsilon>0$. It $T(x) = 0$ then $T(y) = 0$ for each $y \sqsubseteq x$ by positivity and orthogonal additivity of $T$, and there is nothing to prove. Assume $T(x) \neq 0$. By positivity of $T$, the net $(\bigvee_{y\in\pi}T(y))_{\pi\in\Pi_{x}}$ is decreasing. Then we can write
$$
\lambda_{T}(x)=\lim_{\pi\in\Pi_{x}}\bigvee_{y\in\pi} T(y) = 0.
$$
Using the order continuous of $F$, we find $\pi = \{x_i: \, i=1, \ldots, 2n\} \in \Pi_x$, so that
$$
\alpha = \Big\|\bigvee_{i=1}^{2n} T(x_i)\Big\|\leq\frac{\varepsilon^{2}}{2\|T(x)\|}.
$$
Now we put $z_{i}=T(x_{i})$ for $1\leq i\leq 2n$. Note that $K = \sum_{i=1}^{2n} \|z_i\| = \|T(x)\|$.
Choose by Lemma~\ref{le:1kkkk1} a permutation $\tau:\{1,\dots,2n\}\to\{1,\dots,2n\}$ so that
$$
\Big\|\sum_{i=1}^{2n}(-1)^{i}z_{\tau(i)}\Big\|\leq\sqrt{2\alpha K} \leq \varepsilon.
$$
Then for $y_1 = \sum_{k=1}^n x_{\tau(2k-1)}$ and $y_2 = \sum_{k=1}^n x_{\tau(2k)}$ we have $x = y_1 \sqcup y_2$ and
$$
\|T(y_1) - T(y_2)\| = \Big\|\sum_{i=1}^{2n}(-1)^{i}z_{\tau(i)}\Big\| < \varepsilon.
$$
\end{proof}

\section{Pseudo narrow abstract Uryson operators}
\label{sec8}

Let $E,F$ be vector lattices. An abstract Uryson operator $T:E\to F$ is called \emph{disjointness preserving} if $T(x) \bot T(y)$ for all $x,y\in E$ with $x\bot y$. As we will see later (see Theorem~\ref{RTRO}), the set of all disjointness preserving abstract Uryson operators is solid. In particular, an abstract Uryson operator $T:E\to F$ is disjointness preserving if and only if $|T|$ is.

\begin{definition}
Let $E,F$ be vector lattices. An abstract Uryson operator $T\in\mathcal{U}(E,F)$ is called \emph{pseudo-narrow} if there is no disjointness preserving abstract Uryson operator $S\in\mathcal{U}(E,F)$ so that $0 < S\leq|T|$.
\end{definition}

\begin{thm} \label{thm:lambdapseudo}
Let $E,F$ be vector lattices, $F$ Dedekind complete and $E$ atomless. If a positive abstract Uryson operator $T:E\to F$ is $\lambda$-narrow then $T$ is pseudo-narrow.
\end{thm}

\begin{proof}
Suppose $T$ is not pseudo-narrow. Let $S\in\mathcal{U}(E,F)$ be a disjointness preserving operator, $0< S\leq T$ and $x\in E$, such that $S(x)>0$. Then for every representation $x=\bigsqcup\limits_{i=1}^{n}x_{i}$ we have
$$
S(x) = \bigsqcup\limits_{i=1}^n S(x_i) = \bigvee_{i=1}^n S(x_i) \leq \bigvee_{i=1}^n T(x_i).
$$
Hence $\lambda_{T}(x)\geq S(x)>0$ and $T$ is not $\lambda$-narrow.
\end{proof}

The converse assertion will be proved under the additional assumption of the lateral continuity of $T$. We need the some preliminary lemmas. Recall that a vector lattice $E$ is said to have the \emph{projection property} if every band of $E$ is a projection band.

\begin{lemma} \label{le:gemor}
Let $E$ be an atomless vector lattice with the projection property, and $F$ a Dedekind complete vector lattice. Let $T\in\mathcal{U}^+(E,F)$ be a laterally continuous abstract Uryson operator, $e\in E$ and $f\in F^+$. Let $\varphi:\mathfrak{F}_e\to\mathfrak{F}_f$ be a Boolean homomorphism such that $\varphi(x)\leq T(x)$ for each $x\in\mathfrak{F}_e$. Then there exists a disjointness preserving abstract Uryson operator $S\in\mathcal{U}^+ (E,F)$ such that $S \leq T$ and $S(x) = \varphi(x)$ for each $x\in\mathfrak{F}_e$.
\end{lemma}

\begin{proof}
For each $\pi=(x_{i})_{i=1}^{n}\in\Pi_{e}$ we set $L_{\pi}=\text{span}\{x_{i}:\,1\leq i\leq n\}$.
We note that $\pi_{1}\leq\pi_{2}$
in $\Pi_{e}$ implies $L_{\pi_{1}}\subseteq L_{\pi_{2}}$.  In particular,
$(L_{\pi})_{\pi\in\Pi_{e}}$ is an increasing net with respect to the inclusion.
The following linear subspace is a sublattice of $E$
$$
G=\bigcup_{\pi\in\Pi_{e}} L_{\pi} = \text{span} \, \mathfrak{F}_e
$$
(in other words, $G$ is the set of all $e$-step functions). For each $\pi=(x_i)_{i=1}^n \in \Pi_e$ we define the orthogonally additive operator $S_{\pi}:L_{\pi}\to F$ which extends the equality $S_{\pi}(x_i) = \varphi(x_i)$ for $i=1,\dots,n$ to $L_{\pi}$ by the following rule
\begin{equation} \label{eq:ffhdj7}
S_{\pi}\Big(\sum_{i=1}^n \lambda_i x_i \Big)=\sum_{i=1}^{n}|\lambda_{i}|S_{\pi}(x_i).
\end{equation}
Observe that, if $y \in L_\pi$ with $|y| \leq e$ then $y = \sum_{i=1}^n \lambda_i x_i$ with $|\lambda_i| \leq 1$ for all $i$. Hence by \eqref{eq:ffhdj7},
\begin{equation} \label{eq:ffhdj75}
S_\pi (y) \leq \sum_{i=1}^n S_\pi (x_i) = \sum_{i=1}^n \varphi(x_i) = T(e).
\end{equation}
Since $\varphi$ is a Boolean homomorphism on $\mathfrak{F}_e$, we have that if $\pi_{1}\leq\pi_{2}$ in $\mathfrak{F}_e$ then $S_{\pi_{2}}|_{L_{\pi_{1}}}=S_{\pi_{1}}$.
Thus, we can define the following orthogonally additive operator $S_{1}:G\to F$
\begin{equation} \label{eq:ffhdj8}
S_1(x)=\lim_{\pi\in\Pi_{e}}S_{\pi}(x),\,\,x\in G.
\end{equation}

Moreover, since $\varphi$ is a Boolean homomorphism, $S_1$ is a disjointness preserving operator.

Assume now that $m \in \mathbb N$ and $y \in E$ with $|y| \leq me$. Then $m^{-1} y \leq e$ and for each $\pi \in \Pi_e$ by \eqref{eq:ffhdj75} and \eqref{eq:ffhdj7}, $m^{-1} S_\pi(y) = S_\pi(m^{-1}y) \leq T(e)$, that is, $S_\pi(y) \leq m T(e)$. By arbitrariness of $\pi \in \Pi_e$ and \eqref{eq:ffhdj8}, we obtain
\begin{equation} \label{eq:ffhdj85}
(\forall m \in \mathbb N)(\forall y \in G) \, \Bigl( |y| \leq me \,\,\, \Rightarrow \,\,\, S_1(y) \leq m T(e) \Bigr).
\end{equation}

Let $I_e = \bigl\{x \in E: \, (\exists m \in \mathbb N) (|x| \leq me) \bigr\}$ be the order ideal generated by $e$. Fix any $x \in I_e$. Then $|x| \leq me$ for some $m \in \mathbb N$. Then for any $y \in G$ with $|y| \leq |x|$ we have $|y| \leq me$, and $S_1(y) \leq m T(e)$ by \eqref{eq:ffhdj85}. Thus, the set $\{S_{1}y:\,\,|y|\leq|x|,\,y\in G\}$ is order bounded in $F$ by $mT(e)$. Since $F$ is Dedekind complete, there exists
\begin{equation} \label{eq:ffhdj86}
S_2(x) = \sup\{S_1(y):\,\,|y|\leq|x|,\,y\in G\}
\end{equation}
with $S_2(x) \leq m T(e)$. Moreover, if $A \subseteq I_e$ is an order bounded set by $me$ then $S_2(y) \leq m T(e)$ for all $y \in A$.

We show that formula \eqref{eq:ffhdj86} defines a positive abstract Uryson operator $S_2: I_e \to F$. By the above, $S_2$ is an order bounded map. Since $S_1 \geq 0$, we have that $S_2(x) \geq 0$ for all $x \in I_e$. So, it remains to show the orthogonal additivity of $S_2$. Let $x_1, x_2 \in I_e$ and $x_1 \bot x_2$. Then
$$
S_2(x_1 + x_2)=\sup\{S_1 (y):\,|y| \leq |x_1 + x_2| = |x_1| + |x_2|, \, y \in G\}.
$$
Let $y\in G$ be any element with $|y|\leq|x_1| + |x_2|$. By the Riesz decomposition property \cite[Theorem~1.15]{Al}, there exist $y_1, y_2 \in G$ so that $y = y_1 + y_2$ and $|y_i|\leq|x_i|$ for $i \in \{1,2\}$. In particular, $y_1 \bot y_2$. Then
$$
S_1 (y) = S_1(y_1 + y_2) = S_1 (y_1) + S_1 (y_2) \leq S_2 (x_1) + S_2 (x_2).
$$
Passing to the supremum, we obtain $S_2(x_1 + x_2)\leq S_2(x_1) + S_2(x_2)$.

On the other hand, if $y_1, y_2 \in G$ satisfy $|y_i| \leq |x_i|$ for $i \in \{1,2\}$ then $y_1 \bot y_2$ and hence
$$
S_1(y_1) + S_1(y_2) = S_1(y_1 + y_2) \leq S_2(x_1 + x_2).
$$
Passing to the supremum first over $y_1$ and then over $y_2$, we get $S_2(x_1) + S_2(x_2) \leq S_2(x_1 + x_2)$. Thus, $S_2: I_e \to F$ is a positive abstract Uryson operator.

Let $T|_{I_e}$ be the restriction of $T$ to the sublattice $I_e$. Set $S_3 = T|_{I_e} \wedge S_{2}$. Using the fact that the set of all laterally continuous abstract Uryson operators from the vector lattice $I_e$ to the Dedekind complete vector lattice $F$ is a band in $\mathcal{U}(E,F)$ \cite[Proposition~3.8]{Maz-1}, we have that $S_3: I_e \to F$ is a positive laterally continuous abstract Uryson operator.

Denote by $B_e$ the band in $E$ generated by $e$. Now we extend $S_3$ from $I_e$ to a positive laterally continuous abstract Uryson operator $S_4: B_e \to F$. For any $x \in B_e$ we set
\begin{equation} \label{ccn7}
S_4(x) = \sup\{S_3(y):\,\, y \sqsubseteq x, \, y \in I_e\}.
\end{equation}
The supremum exists because $F$ is Dedekind complete and $S_3(y) \leq T(y) \leq T(x)$ for each $y \sqsubseteq x$, $y \in I_e$ by the inequality $S_3 \leq T|_{I_e}$. By the above, $T_4(x) = T_3(x)$ for all $x \in I_e$ and $S_4 (x) \leq T(x)$ for all $x \in B_e$.

Show that $S_4$ is orthogonally additive. Let $x_1, x_2 \in B_e$ and $x_1 \bot x_2$. Then
$$
S_4(x_1 + x_2) = \sup \{S_3 (y): \, y \sqsubseteq x_1 + x_2, \, y \in I_e\}.
$$
Let $y \in I_e$ satisfy $y \sqsubseteq x_1 + x_2$. Since $I_e$ is an ideal in $E$, we have that $y_i = (y^+ \wedge x_i^+) - (y^- \wedge x_i^-) \in I_e$, $i \in \{1,2\}$. Observe that $y_i \sqsubseteq x_i$ and $y_1 \bot y_2$, hence,
$$
S_3(y) = S_3(y_1) + S_3(y_2) \leq S_4(x_1) + S_4(x_2).
$$
Passing to the supremum, we obtain $S_4(x_1 + x_2) \leq S_4(x_1) + S_4(x_2)$.

On the other hand, if $y_1, y_2 \in I_e$ satisfy $y_i \sqsubseteq x_i$ for $i \in \{1,2\}$ then $y_1 \bot y_2$ and hence
$$
S_3(y_1) + S_3(y_2) = S_3(y_1 + y_2) \leq S_4(x_1 + x_2).
$$
Passing to the supremum firstly over $y_1$ and secondly over $y_2$, we get $S_4(x_1) + S_4(x_2) \leq S_4(x_1 + x_2)$. Thus, $S_4: B_e \to F$ is a positive abstract Uryson operator.

Now we show that $S_4$ is laterally continuous. Observe that, if $T_1: E_1 \to F_1$ is a positive abstract Uryson operator then the lateral continuity of $T_1$ at $z \in E_1$ is equivalent to the following property: if $(z_\alpha)$ is a $\sqsubseteq$-increasing net of fragments of $z$ with $z_\alpha \stackrel{\rm o}{\longrightarrow} z$ then $T_1(z) \leq \sup_\alpha T_1(z_\alpha)$. Indeed, the inequality $T_1(z) \geq \sup_\alpha T_1(z_\alpha)$ is also true because $T_1(z_\alpha) \leq T_1(z)$ for all indices $\alpha$, and $T_1(z_\alpha) \stackrel{\rm o}{\longrightarrow} \sup_\alpha T_1(z_\alpha)$ because the net $T_1(z_\alpha)$ increases.

So, let $x \in B_e$, let $(x_\alpha)$ be a net, $x_\alpha \sqsubseteq x_\beta \sqsubseteq x$ for $\alpha < \beta$ and $x_\alpha \stackrel{\rm o}{\longrightarrow} x$. Fix any $y \in I_e$ with $y \sqsubseteq x$ and set $y_\alpha = (y^+ \wedge x_\alpha^+) - (y^- \wedge x_\alpha^-)$. Then $y_\alpha \sqsubseteq x_\alpha$, $y_\alpha \sqsubseteq y$ for all $\alpha$, $y_\alpha \stackrel{\rm o}{\longrightarrow} y$ and $y_\alpha \in I_e$. By the lateral continuity of $S_3$,
$$
S_3(y) = \sup_\alpha S_3(y_\alpha) \leq \sup_\alpha \sup \{S_3 (z): \, z \sqsubseteq x_\alpha, \, z \in I_e\} = \sup_\alpha S_4(x_\alpha).
$$
By arbitrariness of $y$,
$$
S_4(x) = \sup\{S_3(y):\,\, y \sqsubseteq x, \, y \in I_e\} \leq \sup_\alpha S_4(x_\alpha),
$$
and the lateral continuity of $S_4$ is proved.

Finally, we define an abstract Uryson operator $S:E\to F$ by $S = S_4 \circ P_e$ where $P_e$ is the band projection of $E$ onto $B_e$.

Show that $S$ is a disjointness preserving operator. Assume, on the contrary, that $w = Sx \wedge Sy > 0$ for some $x,y \in E$ with $x \bot y$. Set $u = P_e x$, $v = P_e y$. Then $u \bot v$ and $S_4(u) \wedge S_4(v) = w > 0$. Using \eqref{ccn7}, we find $u_1 \sqsubseteq u$ and $v_1 \sqsubseteq v$ so that $u_1, v_1 \in I_e$ and $w_1 = S_3(u_1) \wedge S_3(v_1) > 0$. Since $S_2 \geq S_3$, we obtain $w_2 = S_2(u_1) \wedge S_2(v_1) \geq w_1 > 0$. By \eqref{eq:ffhdj86}, there are $u_2, v_2 \in G$ such that $|u_2| \leq |u_1$, $|v_2| \leq |v_1$ and $w_3 = S_1(u_2) \wedge S_1(v_2) > 0$. Since $|u_2| \leq |u_1| \leq |u|$ and analogously $|v_2| \leq |v|$, one has that $u_2 \bot v_2$. This is impossible because $S_1$ is a disjointness preserving operator.

It remains to observe that $S(x) = S_4(P_e x) \leq T(P_e x) \leq T(x)$ for all $x \in E$.
\end{proof}

\section{Pseudo-embeddings and a generalization of Rosenthal's decomposition theorem to abstract Uryson operators}

Let $E$ be a vector lattice. For an arbitrary index set $J$ a series $\sum_{j \in J} x_j$ of elements $x_j \in E$ is called \textit{order convergent} and the family $(x_j)_{j \in J}$ is called \textit{order summable} if the net $(y_s)_{s \in J^{< \omega}}$, $y_s = \sum_{j \in s} x_j$ order converges to some $y_0 \in E$, where $J^{< \omega}$ is the net of all finite subsets $s \subseteq J$ ordered by inclusion. In this case $y_0$ is called the \textit{order sum of the series} $\sum_{j \in J} x_j$ and we write $y_0 = \sum_{j \in J} x_j$. A series $\sum_{j \in J} x_j$ is called \textit{absolutely order convergent} and the family $(x_j)_{j \in J}$ is called \textit{absolutely order summable} if the series $\sum_{j \in J} |x_j|$ order converges.

Let $A$ be a solid subset of a vector lattice $E$, that is, for every $x \in E$ and $y \in A$ the inequality $|x| \leq |y|$ implies $x \in A$. We denote by $\mbox{Abs}\,(A)$ the set of all sums of absolutely order convergent series $\sum_{j \in J} x_j$ of elements $x_j \in A$. For any subset $B \subseteq E$ by $\mbox{Band}(A)$ we denote the least band in $E$ containing~$A$.

\begin{lemma}[Theorem~1.25, \cite{PR}] \label{RTL}
Let $A$ be a solid subset of a Dedekind complete vector lattice $E$. Then
\begin{enumerate}
  \item[$(i)$] $A^d =  \{x \in E: \,\,\, \mbox{for all} \,\, y \in A, \,\,\,\,\, 0 \leq y \leq |x| \,\,\, \mbox{implies} \,\,\, y = 0  \}$;
  \item[$(ii)$] ${\rm Band} \,(A) = {\rm Abs} \,(A)$.
\end{enumerate}
In particular, $E = {\rm Abs}\,(A) \oplus A^d$ is a decomposition into mutually complemented bands.
\end{lemma}

Let $E$ and $F$ be vector lattices. By $\mathcal U_{dpo}(E,F)$ we denote the set of all disjointness preserving abstract Uryson operators $T:E\to F$.

\begin{definition} \label{PA}
Let $E,F$ be vector lattices with $F$ Dedekind complete. An operator $T \in \mathcal U(E,F)$ is called \textit{pseudo embedding} if there exists an absolutely order summable family $(T_j)_{j \in J}$ in $\mathcal U_{dpo}(E,F)$ such that $T = \sum_{j \in J} T_j$.
\end{definition}

The name of ``pseudo embedding'' is due to Rosenthal's theorem concerning operators on $L_1$ which asserts that a non-zero operator is a pseudo embedding if and only if it is a near isometric embedding when restricted to a suitable $L_1(A)$-subspace \cite[Theorem~7.39]{PR}.

The set of all pseudo embeddings from $\mathcal U(E,F)$ will be denoted $\mathcal U_{pe}(E,F)$. Thus, $\mathcal U_{pe}(E,F) = \mbox{Abs} \,  (\mathcal U_{dpo}(E,F)  )$ by definition. The set of all pseudo narrow operators $T \in \mathcal U(E,F)$ will be denoted $\mathcal U_{pn}(E,F)$.

\begin{thm} \label{RTRO}
Let $E, F$ be vector lattices with $F$ Dedekind complete. Then
\begin{enumerate}
  \item[$(i)$] $\mathcal U_{dpo}(E,F)$ is solid in $\mathcal U(E,F)$;
  \item[$(ii)$] ${\rm Band}\, \bigl(\mathcal U_{dpo}(E,F) \bigr) = \mathcal U_{pe}(E,F)$;
  \item[$(iii)$] $\mathcal U_{dpo}(E,F)^d = \mathcal U_{pn}(E,F)$;
  \item[$(iv)$] $\mathcal U_{pe}(E,F)$ and $\mathcal U_{pn}(E,F)$ are mutually complemented bands, hence \\ $\mathcal U(E,F) = \mathcal U_{pe}(E,F) \oplus \mathcal U_{pn}(E,F).$
\end{enumerate}
\end{thm}

\begin{proof}
By Lemma~\ref{RTL}, it is enough to prove (i). Suppose $S \in \mathcal U(E,F)$, $T \in \mathcal U_{dpo}(E,F)$ and $|S| \leq |T|$. Our goal is to prove that $S \in \mathcal U_{dpo}(E,F)$. First we prove that $|T| \in \mathcal U_{dpo}(E,F)$. Let $x,y \in E$ and $x\bot y$. By (1) of Theorem~\ref{fjjjjjg},
$$
|T|(x) = \sup \{T(u) - T(v): \,\, x = u \sqcup v\};
$$
$$
|T|(y) = \sup\{T(w) - T(z): \,\, y = w \sqcup z\}.
$$

Let $x = u \sqcup v$ and $y = w \sqcup z$. Since $T$ is disjointness preserving, $|T(u) - T(v)| = |T(u) + T(v)| = |T(u+v)| = |T(x)|$. Analogously, $|T(w) - T(z)| = |T(y)|$. Hence,
$$
|T(u) - T(v)| \wedge |T(w) - T(z)| = |T(x)| \wedge |T(y)| = 0.
$$
Passing to the supremum, we obtain $|T|(x)\bot|T|(y)$. Thus, $|T| \in \mathcal U_{dpo}(E,F)$.

Now we prove that $S \in L_{dpo}(E,F)$. Let $x_1, x_2 \in E$, $ x \bot \, y$. Since $|S x_i| \leq |S| \, |x_i| \leq |T| \, |x_i|$, $i = 1,2,$ we obtain that
$$
0 \leq |S x_1| \wedge |S x_2| \leq |T| \, |x_1| \wedge |T| \, |x_2| = 0.
$$
Thus, $S x_1 \bot \, S x_2$.
\end{proof}

By \cite[Proposition~3.8]{Maz-1}, the set $\mathcal U^{lc}(E,F)$ of all laterally continuous abstract Uryson operators from $E$ to $F$ is a band in $\mathcal U(E,F)$. Since the intersection of bands is a band, we obtain the following version of Theorem~\ref{RTRO} for order continuous operators.

\begin{cor} \label{RTROcor}
Let $E, F$ be vector lattices with $F$ Dedekind complete. Then
\begin{enumerate}
  \item[$(i)$] the set $\mathcal U_{dpo}^{lc}(E,F)$ of all disjointness preserving laterally continuous operators is solid in $\mathcal U^{lc}(E,F)$;
  \item[$(ii)$] $\mbox{\rm Band}\, \bigl(\mathcal U_{dpo}^{lc}(E,F) \bigr) = \mathcal U_{pe}^{lc}(E,F)$;
  \item[$(iii)$] $\mathcal U_{dpo}^{lc}(E,F)^d = \mathcal U_{pn}^{lc}(E,F)$;
  \item[$(iv)$] the sets $\mathcal U_{pe}^{lc}(E,F)$ and $\mathcal U_{pn}^{lc}(E,F)$ are mutually complemented bands, hence $\mathcal U^{lc}(E,F) = \mathcal U_{pe}^{lc}(E,F) \oplus \mathcal U_{pn}^{lc}(E,F).$
\end{enumerate}
\end{cor}

Here, $\mathcal U_{dpo}^{lc}(E,F)$, $\mathcal U_{pe}^{lc}(E,F)$ and $\mathcal U_{pn}^{lc}(E,F)$   denote the corresponding intersections of $\mathcal U_{dpo}(E,F)$, $\mathcal U_{pe}(E,F)$ and $\mathcal U_{pn}(E,F)$ with $\mathcal U^{lc}(E,F)$.

\section{Boolean maps}

We need some information about Boolean maps and homomorphisms.

\begin{definition}
Let $A,B$ be lattices (not necessarily vector spaces). A map
$\psi:A\to B$ is said to be
\begin{enumerate}
  \item[1)] $\vee$-preserving if $\psi(x\vee y)=\psi(x)\vee\psi(y)$ for all $x,y\in A$.
  \item[2)] $\wedge$-preserving if $\psi(x\wedge y)=\psi(x)\wedge\psi(y)$ for all $x,y\in A$.
  \item[3)] a lattice homomorphism provided it is both $\vee$-preserving and $\wedge$-preserving.
\end{enumerate}
\end{definition}
If, moreover, $A$ and $B$ are Boolean algebras then in each of the
above definitions we additionally claim that $\psi(\mathbf{0}_{A})=\mathbf{0}_{B}$ and $\psi(\mathbf{1}_{A})=\mathbf{1}_{B}$. In this case we insert the word ``Boolean'' (a Boolean $\vee$-preserving map; a Boolean $\wedge$-preserving map; a Boolean homomorphism).

We need the following two known facts.

\begin{thm}[Monteiro's theorem \cite{PR},Theorem~10.33] \label{th:Mont}
Let $A,B$ be Boolean algebras with $B$ Dedekind complete and $\varphi:A\to B$
be a Boolean $\vee$-preserving map. Then every Boolean homomorphism $\psi_{0}:A_{0}\to B$
of a Boolean subalgebra $A_{0}\subseteq A$ with $\psi_{0}(x)\leq\varphi(x)$  for each $x\in A_0$ can be extended to a Boolean homomorphism $\psi:A\to B$ with $\psi(x)\leq\varphi(x)$ for each
$x\in A$.
\end{thm}

Recall that for the band projection $P_e x$ of an element $x \in E$ to the band generated by $e \in E^+$ we have the formula $P_e x = \bigvee_{n=1}^\infty (x \wedge ne) \in \mathfrak{F}_x$ \cite[Theorem~3.13]{Al}.

\begin{lemma}[\cite{MMP}, \cite{PR}, Lemma~10.38] \label{le:MMP222}
Let $F$ be a Dedekind  complete vector lattice, $f\in F^+$. Then the formula
\begin{equation} \label{eq:xxsre7}
\mathbf{1}_{f}(y) = f - P_{(f-y)^+} f
\end{equation}
defines a lattice homomorphism $\mathbf{1}_{f}:F^+ \to \mathfrak{F}_f$ such that $\mathbf{1}_{f}(y)\leq y$ for all $y\in F_{+}$, $\mathbf{1}_{f}(0)=0$ and $\mathbf{1}_{f}(f)=f$.
\end{lemma}

Now we are ready to prove the next auxiliary statement.

\begin{lemma} \label{le:djjjdjd75}
Let $E,F$ be vector lattices with $F$ Dedekind  complete. Let $T:E\to F$
be an abstract Uryson  operator and $e\in E$ be such that $f=\lambda_{T}(e)>0$. Then the formula
$\varphi(x) = \mathbf{1}_{f} \bigl(\lambda_{T}(x) \bigr)$, where $\mathbf{1}_{f}$ is the map defined by \eqref{eq:xxsre7}, defines a Boolean $\vee$-preserving map $\varphi:\mathfrak{F}_e \to \mathfrak{F}_f$ such that $\varphi(x)\leq|T|(x)$ for all $x \in \mathfrak{F}_e$.
\end{lemma}
\begin{proof}
By Lemma~\ref{le:MMP222}, $\varphi(0) = \mathbf{1}_{f}(\lambda_{T}(0)) = \mathbf{1}_{f}(0)=0$ and
$\varphi(e) = \mathbf{1}_{f} (\lambda_{T}(e)) = \mathbf{1}_{f}(f)=f$. Fix any $x \in \mathfrak{F}_e$. Observe that, for any $\pi \in \Pi_x$ and any $u \in \pi$ one has $|T|(u) \leq |T|(x)$. Hence, $\bigvee_{u \in \pi} |T|(u) \leq |T|(x)$. Then we may write
\begin{equation*}
\begin{split}
\varphi (x) &= \mathbf{1}_f \bigl( \lambda_T(x) \bigr) \stackrel{\mbox{\tiny by Lemma~\ref{le:MMP222}}}{\leq} \lambda_T(x) = \bigwedge_{\pi \in \Pi_x} \bigvee_{u \in \pi} |T (u)| \\
&\leq \bigwedge_{\pi \in \Pi_x} \bigvee_{u \in \pi} |T| (u) \leq \bigwedge_{\pi \in \Pi_x} |T| (x) = |T| (x).
\end{split}
\end{equation*}

Now we prove that $\varphi$ is a $\bigvee$-preserving map. Show firstly that $\varphi(x\sqcup y)=\varphi(x)\vee\varphi(y)$. Let $x,y\in\mathfrak{F}_e$ and $x\bot y$. Then we have
\begin{equation*}
\begin{split}
\lambda_T( x \sqcup y ) &= \lim_{\pi \in \Pi_{x \sqcup y}} \, \bigwedge_{\pi' \geq \pi} \,\, \bigvee_{u \in \pi'} |T(u)| = \lim_{\Pi_{x \sqcup y} \ni \pi \geq \{x,y\}} \, \bigwedge_{\pi' \geq \pi} \,\, \bigvee_{u \in \pi'} |T(u)|\\
&=\lim_{\pi \in \Pi_x, \, \tau \in \Pi_y} \, \bigwedge_{\pi' \geq \pi, \, \tau' \geq \tau} \,\,
\bigvee_{u \in \pi' \cup \tau'} |T(u)|\\
&= \lim_{\pi \in \Pi_x, \, \tau \in \Pi_y} \, \bigwedge_{\pi' \geq \pi, \, \tau' \geq \tau} \Bigl( \bigvee_{u \in \pi'} |T(u)| \vee \bigvee_{v \in \tau'} |T(v)| \Bigr)\\
&= \lim_{\pi \in \Pi_x, \, \tau \in \Pi_y} \Bigl( \Bigl(\bigwedge_{\pi' \geq \pi} \,\, \bigvee_{u \in \pi'} |Tu| \Bigr) \vee \Bigl( \bigwedge_{\tau' \geq \tau} \,\, \bigvee_{v \in \tau'} |Tv| \Bigr) \Bigr)\\
&= \lim_{\pi \in \Pi_x} \bigwedge_{\pi' \geq \pi}\,\, \bigvee_{u \in \pi'} |Tu| \vee \lim_{\tau \in \Pi_y} \bigwedge_{\tau' \geq \tau} \,\, \bigvee_{v \in \tau'} |Tv|\\
&= \bigwedge_{\pi \in \Pi_x} \,\, \bigvee_{u \in \pi} |Tu| \vee \bigwedge_{\tau \in \Pi_y} \,\, \bigvee_{v \in \tau} |Tv| = \lambda_T(x) \vee \lambda_T(y).
\end{split}
\end{equation*}

\begin{equation*}
\begin{split}
\mbox{Hence,} \,\,\, \varphi( x \sqcup y ) &= \mathbf{1}_e \bigl( \lambda_T ( x \sqcup y )
\bigr) = \mathbf{1}_e \bigl( \lambda_T(x) \vee \lambda_T(y) \bigr)
\\&
= \mathbf{1}_e \bigl( \lambda_T(x) \bigr) \vee \mathbf{1}_e \bigl(
\lambda_T(y) \bigr) = \varphi(x) \vee \varphi(y). \,\,\,\,\,\,\,\,\,\,\,\,\,\,\,\,\,\,\,\,\,\,\,\,\,\,\,\,\,\,\,\,\,\,\,\,
\end{split}
\end{equation*}

If $x,y\in\mathfrak{F}_e$ with $x\leq y$ then
$$
\varphi(y)=\varphi \bigl((y-x)\sqcup(x) \bigr) = \varphi(y-x) \vee \varphi(x) \geq \varphi(x).
$$
Therefore $\varphi(x\vee y)\geq\varphi(x)\vee\varphi(y)$ for all $x,y \in \mathfrak{F}_e$.

On the other hand, for all $x,y \in \mathfrak{F}_e$ one has
$$
\varphi(x\vee y) = \varphi \bigl( (x - x \wedge y) \sqcup y \bigr) = \varphi(x - x \wedge y) \vee \varphi(y) \leq \varphi(x) \vee \varphi(y).
$$
Thus $\varphi(x\vee y) = \varphi(x)\vee\varphi(y)$ for every $x,y\in\mathfrak{F}_e$.
\end{proof}

\section{The main results}

\begin{thm} \label{th:ddddjfhhd7}
Let $E$ be an atomless vector lattice with the projection property, $F$ a Dedekind complete vector lattice. Then a positive laterally continuous abstract Uryson operator $T: E \to F$ is $\lambda$-narrow if and only if is pseudo-narrow.
\end{thm}

\begin{proof}
By Theorem~\ref{thm:lambdapseudo} we only need to prove that a pseudo-narrow operator is $\lambda$-narrow. Suppose that $T$ is not $\lambda$-narrow and $e\in E$ is
such that $f=\lambda_{T}(e)>0$. By Lemma~\ref{le:djjjdjd75} we construct a Boolean $\vee$-preserving map $\varphi:\mathfrak{F}_e \to \mathfrak{F}_f$ with the corresponding properties. Let $A=\mathfrak{F}_e$, $B=\mathfrak{F}_f$, $A_{0}=\{0,e\}$,
$\psi_{0}:A_{0}\to B$ be the trivial Boolean homomorphism (i.e. $\psi_{0}(0)=0$ and
$\psi_{0}(e)=f$). Evidently, $\psi_{0}(x)\leq\varphi(x)$ for every $x\in A_{0}$. By Theorem~\ref{th:Mont}, there is a Boolean homomorphism $\psi: \mathfrak{F}_e \to \mathfrak{F}_f$ such that $\psi$ extends $\psi_{0}$ with $\psi(x)\leq\varphi(x)$ for all $x\in\mathfrak{F}_e$. By the choice of $\varphi$, one has that $\varphi(x)\leq T(x)$ for all $x\in\mathfrak{F}_e$. Thus,
$\psi(x)\leq T(x)$ for all $x\in\mathfrak{F}_e$. By Lemma~\ref{le:gemor}, there exists a disjointness preserving abstract Uryson operator $S: E \to F$ such that $0 \leq S \leq T$ and $S(x)=\psi(x)$ for all $x\in\mathfrak{F}_e$. In particular, $S(e)=\psi(e)=\psi_{0}(e)=f>0$
and hence $S>0$. This means that $T$ is not pseudo-narrow.
\end{proof}

As a consequence of the above results, we obtain the following theorem.

\begin{thm} \label{th:mainofwhich}
Let $E,F$ be  vector lattices such that E is atomless and has the projection property,
F an ideal of some order continuous Banach lattice. Then
\begin{enumerate}
  \item[(1)] Every laterally continuous abstract Uryson operator $T: E \to F$ is order narrow if and
only if it is pseudo-narrow.
  \item[(2)] The set $\mathcal U_{on}^{lc}(E,F)$ of all order narrow laterally continuous abstract Uryson operators is a band in the Dedekind complete vector lattice $\mathcal U^{lc}(E,F)$ of all laterally continuous abstract Uryson operator from $E$ to $F$. Moreover, the orthogonal complement to $\mathcal U_{on}^{lc}(E,F)$ equals the band generated by all disjointness preserving laterally continuous abstract Uryson operators from $E$ to $F$, which, in turn, equals the set $\mathcal U_{pe}^{lc}(E,F)$ of all laterally continuous pseudo embeddings.
  \item[(3)] Each laterally continuous abstract Uryson operator $T:E\to F$ is uniquely represented
            in the form $T=T_{D}+T_{N}$ where $T_{D}$ is a laterally continuous abstract Uryson pseudo embedding and $T_{N}$ is a laterally continuous narrow abstract Uryson operator.
\end{enumerate}
\end{thm}
\begin{proof}
Note that by \cite[p.~7]{LTII}, $F$ is Dedekind complete. We prove $(i)$ by the following scheme
$$
T \,\, \mbox{is order narrow} \stackrel{^{{\rm Theorem} \,
\ref{th:narmod}}}{\Longleftrightarrow} |T| \,\, \mbox{is order narrow}
\stackrel{^{{\rm Theorem} \, \ref{th:lamordern}}}{\Longleftrightarrow}
|T| \,\, \mbox{is $\lambda$-narrow}
$$

$$
\stackrel{^{{\rm Theorem} \, \ref{th:ddddjfhhd7}}}{\Longleftrightarrow}
|T| \,\, \mbox{is pseudo narrow} \stackrel{^{{\rm by} \,\, {\rm
definition}}}{\Longleftrightarrow} T \,\, \mbox{is pseudo narrow}.
$$

Items $(2)$ and $(3)$ follow from $(1)$ and Corollary~\ref{RTROcor}.
\end{proof}

\end{document}